\documentclass[12pt]{article}
\usepackage{amsthm}
\usepackage{amsfonts}
\usepackage{epsfig}
\usepackage{pstricks}
\usepackage{hyperref}
\newtheorem{theorem}{Theorem}
\newtheorem{conjecture}[theorem]{Conjecture}

\newtheorem{lemma}[theorem]{Lemma}

\def\myincludegraphics#1{\begin{center}\includegraphics{#1}\end{center}}

\title{Bipartizing fullerenes\thanks{Supported by a CZ-SL bilateral project MEB 091037 and BI-CZ/10-11-004.}}
\author{Zden\v{e}k Dvo\v{r}\'ak\thanks{Department of Applied Mathematics, Faculty of Mathematics and 
           Physics, Charles University, Malostransk\'e n\'am\v{e}st\'{\i}~25, 118~00 Prague, 
           Czech Republic. E-mail: {\tt rakdver@kam.mff.cuni.cz}. 
           Partially supported by Institute for Theoretical Computer Science (ITI),
           project of Ministry of~Education of~Czech Republic, and ERC starting grant CCOSA (grant agreement no. 259385).}\and
        Bernard Lidick\'y\thanks{Department of Applied Mathematics, Faculty of Mathematics 
           and Physics, Charles University, Malostransk\'e n\'am\v{e}st\'{\i}~25, 
           118~00 Prague, Czech Republic. E-mail: {\tt bernard@kam.mff.cuni.cz}.}\and
        Riste \v{S}krekovski\thanks{Department of Mathematics, University of Ljubljana, 
           Jadranska~19, 1111 Ljubljana, Slovenia.  E-mail: {\tt skrekovski@gmail.com}.
           Partially supported by ARRS, Research Program P1-0297. }}
\date{\today}

\begin{document}
\maketitle

\begin{abstract}
A fullerene graph is a cubic bridgeless planar graph with twelve 5-faces such that all
other faces are 6-faces. 
We show that any fullerene graph on $n$ vertices can be bipartized by removing $O(\sqrt{n})$ edges.
This bound is asymptotically optimal.
\end{abstract}


\section{Introduction}

{\em Fullerenes} are carbon-cage molecules comprised of carbon atoms
that are arranged on a sphere with pentagonal and hexagonal faces. 
The icosahedral $C_{60}$, well-known as Buckminsterfullerene
was found by Kroto et al.~\cite{KHBCS}, and later confirmed by experiments by
Kr\"{a}tchmer et al.~\cite{KLFH} and Taylor et al.~\cite{THAK}.
Since the discovery of the first fullerene molecule, the 
fullerenes have been objects of interest to scientists all over the world.

From the graph theoretical point of view, the fullerenes can be viewed as cubic 3-connected graphs embedded into
a sphere with face lengths being $5$ or $6$. Euler's formula implies that each fullerene contains
exactly twelve pentagons, but provides no restriction on the number of hexagons. In fact, it is not difficult to see
that mathematical models of fullerenes with precisely $\alpha$ hexagons exist for all values of $\alpha$ with the
sole exception of $\alpha=1$.  See~\cite{DDE, FM, GC, M} for more information on chemical, physical, and 
mathematical properties of fullerenes. 

The question of stability of fullerene molecules receives a lot of attention. The goal is to
obtain a graph-theoretical property whose value influences the stability. 
Different properties, like the number of perfect matchings~\cite{KKMS} or the independence number~\cite{FL} were
considered.  The property investigated in this paper is how far the graph is
from being a bipartite graph, which was suggested by Do\v{s}li\'c~\cite{D} and
further considered in \cite{DV}. Despite of the effort none of the 
so far considered  parameters works in all cases. Hence more research is still needed.

For a plane graph $H$, let $F(H)$ be the set of the faces of $H$.
Let $H$ be a fullerene graph, and let $K_H$ be the weighted complete graph whose vertices correspond
to the $5$-faces of $H$ and the weight of the edge joining two $5$-faces $f_1$ and $f_2$ is
equal to the distance from $f_1$ to $f_2$ in the dual of $H$.  Let $b(H)$ be the size of the minimum
set $S\subseteq E(H)$ such that $H-S$ is bipartite.  Do\v{s}li\'c and Vuki\v{c}evi\'c~\cite{DV} proved the following:

\begin{theorem}\label{thm-bibmatch}
If $H$ is a fullerene graph, then $b(H)$ is equal to the minimum weight of a perfect matching in $K_H$.
\end{theorem}

A corollary of the above theorem is a polynomial-time algorithm for finding a set of edges $S$ 
whose removal makes the graph bipartite. 

Do\v{s}li\'c and Vuki\v{c}evi\'c~\cite{DV} conjectured that $b(H)=O(\sqrt{|V(H)|})$.  In fact, they
gave the following stronger conjecture.
\begin{conjecture}
If $H$ is a fullerene graph with $n$ vertices, then $b(H)\le \sqrt{12n/5}$.
\end{conjecture}

The main result of this paper is an upper bound on $b(H)$, confirming the weaker version of the conjecture.
\begin{theorem}\label{thm-main}
If $H$ is a fullerene graph with $n$ vertices, then $b(H)=O(\sqrt{n})$.
\end{theorem}

\section{Proof of Theorem~\ref{thm-main}}

Let $H$ be a fullerene graph.  A {\em patch with boundary $o$} is a $2$-connected
subgraph $G\subseteq H$ such that $o\in F(G)$ (usually, we consider $o$ to be the outer face of $G$) and $F(G)\setminus F(H)\subseteq \{o\}$
(but it is possible for the boundary $o$ to be also a face of $G$).  Let $v$ be a vertex incident with $o$.
If $\deg_G(v)=3$, then $v$ is a {\em $3$-vertex (with respect to $o$)},
otherwise $v$ is a {\em $2$-vertex (with respect to $o$)}.  
An edge $e$ incident with
$o$ is a {\em $22$-edge} (resp. {\em a $33$-edge}) if both vertices incident with $e$ are $2$-vertices (resp. $3$-vertices)
with respect to $o$.  If $e$ is neither a $22$-edge nor a $33$-edge, then it is a {\em $23$-edge}.
The {\em description} $D(o)$ of the boundary $o$ is the cyclic sequence
in that $A$ represents a $33$-edge, $B$ represents a $22$-edge, and a maximal consecutive segment of $23$-edges is represented
by the integer giving its length.  For example, the boundary of the patch consisting of a $5$-face and a $6$-face sharing an edge
is described as $BB2BBB2$.

Let $s(o)$ and $t(o)$ be the numbers of $22$-edges and $33$-edges of $o$, respectively, and let $s_2(o)$ be the number of pairs of consecutive $22$-edges of $o$.
Let $p(G)$ be the number of $5$-faces of $G$ distinct from $o$.
The following lemma relates the number of $22$- and $33$-edges; a similar relation was derived by
Kardo\v{s} and \v{S}krekovski~\cite{KS}.

\begin{lemma}\label{lemma-n22}
If $G$ is a patch with the boundary $o$, then $s(o)=6-p(G)+t(o)$.
\end{lemma}
\begin{proof}
Suppose that the length of $o$ is $\ell$.  Let $n=|V(G)|$, $m=|E(G)|$ and let $f$ be the number of faces of $G$.
Since each edge of $G$ is incident with two
faces, $$2m = 6 (f - p(G) - 1) + 5p(G) + \ell,$$  i.e., $\ell=2m+p(G)+6-6f$.  
Note that the number of $2$-vertices is $(\ell+s(o)-t(o))/2$, 
which can be easily seen from the modification of the boundary by adding $s(o)$ and 
deleting $t(o)$ 3-vertices so that there is no 33-edge or 22-edge.
Thus $$2m=3n - (\ell+s(o)-t(o))/2.$$
Substituting for $\ell$, we obtain $$3m=3n+3f - 6 + \frac{6 - p(G) + t(o) - s(o)}{2}.$$
By Euler's formula, $m=n+f-2$, thus $6 - p(G) + t(o) - s(o)=0$ and the claim of the lemma follows.
\end{proof}

A patch $G$ with the boundary $o$ is a {\em fat worm} if $p(G)=0$, the subgraph of $G$ induced by $V(G)\setminus V(o)$
is a path $P$, and the edges of $E(G)\setminus E(P)$ incident with each two consecutive inner vertices of $P$ are not incident to a common face of $G$.  See Figure~\ref{fig:worms}(a).
Note that in this case, the description of $o$ is 
\begin{itemize}
\item $BB2B(2k+2)BB2B(2k+2)$ if $P$ has length $2k+1$ and 
\item $BB2B(2k+2)B2BB(2k+4)$ if $P$ has length $2k+2$.  
\end{itemize}
We consider the patch with exactly one
vertex not incident with $o$ (and boundary $BB2BB2BB2$) to be a fat worm as well (in this case, $P$ has length $0$).
The patch $G$ is a {\em slim worm} if $p(G)=0$, $V(G)=V(o)$ and $t(o)=0$. Geometrically, it is a straight
line of hexagons, see Figure~\ref{fig:worms}(b). 
Note that $D(o)=BBB(2k)BBB(2k)$ for some $k$ (or $D(o)=BBBBBB$, when
$k=0$ and $o$ is a $6$-face).  The patch $G$ is a {\em worm}
if it is a fat worm or a slim worm.  The {\em shell} is the patch $G$ with boundary $o$ such that $p(G)=0$ and $D(o)=BB4BB4BB4$ (having $4$ internal vertices).
See Figure~\ref{fig:worms}(c).

\begin{figure}
\myincludegraphics{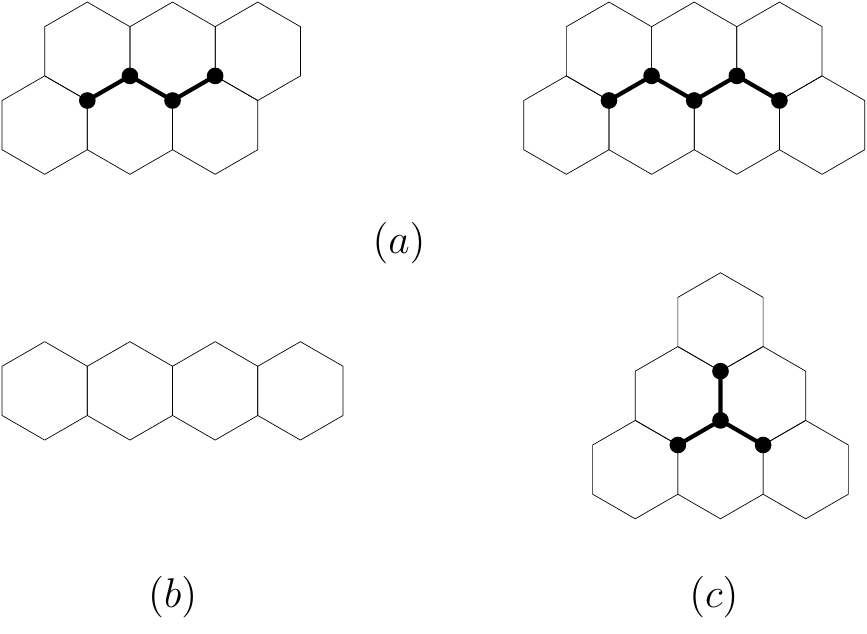}
\caption{A fat worm, a slim worm and the shell. \label{fig:worms}}
\end{figure}

An {\em $\ell$-chord} of a cycle $C$ in a patch $G$ is a path of length $\ell$ with distinct endvertices belonging to $V(C)$ such that
the inner vertices and edges of the path do not belong to $C$.  We say a chord instead of a $1$-chord.
Consider an $\ell$-chord $Q$ of the boundary $o$ of a patch $G$.  Let $G_1$ and $G_2$ be the two patches into that $Q$ splits $G$ (i.e., the
subgraphs such that $G_1\cup G_2=G$, $G_1\cap G_2=Q$ and $G_1\neq Q\neq G_2$), and $o_1$ and $o_2$ their boundaries.
We say that $Q$ {\em splits off a face} if $G_1=o_1$ or $G_2=o_2$.
The patch $G$ is {\em decomposable} if it contains a {\em simplifying cut}, that is

\begin{itemize}
\item an $\ell$-chord $Q$ of $o$ with $\ell\le 3$ such that $t(o_1)+t(o_2)<t(o)$, or
\item two $4$-chords $Q_1=v_0v_1v_2v_3v_4$ and $Q_2=w_0w_1w_2w_3w_4$ such that
$v_0w_0$, $v_2w_2$ and $v_4w_4$ are edges of $G$. See Figure~\ref{fig:24chords}.
\end{itemize}
Otherwise, we call $G$ {\em indecomposable}.  We say that $G$ is a {\em normal patch} if
$G$ is indecomposable, no $5$-face of $G$ distinct from $o$ shares an edge with $o$
and $G$ is neither a worm nor a shell.

\begin{figure}
\myincludegraphics{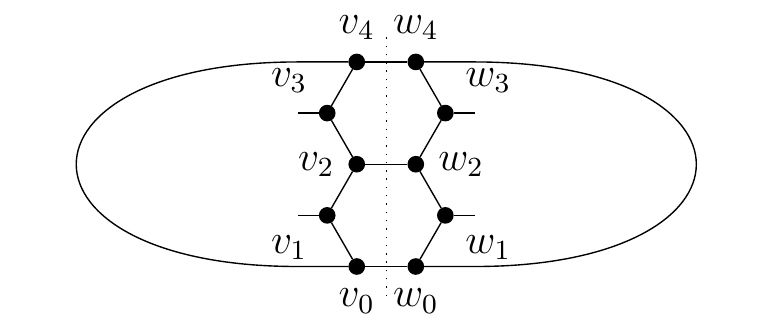}
\caption{Two 4-chords. \label{fig:24chords}}
\end{figure}

\begin{lemma}\label{lemma-nochord}
Let $G$ be a normal patch with boundary $o$ and $Q$ an $\ell$-chord of $o$, with $\ell\le 3$.
Then $\ell\ge 2$ and $Q$ splits off a face.  Furthermore, the number of $33$-edges incident with the
endvertices of $Q$ is most $\ell-2$.
\end{lemma}
\begin{proof}
Let $G_1$ and $G_2$ with boundaries $o_1$ and $o_2$, respectively, be the patches to that $Q$ splits $G$.
Let $Q=q_0q_1\ldots q_{\ell}$ and $o_2=q_0v_1v_2\ldots v_aq_{\ell}q_{\ell-1}\ldots q_0$.

Suppose first that $\ell=1$.  Since $G$ is not a slim worm, there exists an edge $e\in E(G)$ that either is
a $33$-edge of $o$ or is incident with a vertex in $V(G)\setminus V(o)$.  Let us choose the chord $Q$ and the patches $G_1$ and $G_2$
so that $e\in E(G_2)$ and $G_2$ is minimal.  As $G$ is indecomposable, each $33$-edge of $G$ is also a $33$-edge in $G_1$ or $G_2$.
It follows that $v_1$ and $v_a$ are $2$-vertices, and since the internal face incident with $q_0v_1$ has length six,
$v_2$ and $v_{a-1}$ must be adjacent.  Since $e\in E(G_2)$, we have $G_2\neq o_2$; hence, $v_2v_{a-1}$ is a chord of $o$.
The chord $v_2v_{a-1}$ splits $G$ to patches $G'_1$ and $G'_2$ with $G'_2\subset G_2$.  However, this contradicts the choice
of $Q$, since it is easy to see that $e\in E(G')$.  We conclude that $o$ is an induced cycle.

Suppose now that $\ell=2$.  By symmetry between $G_1$ and $G_2$, we may assume that $q_1$ is a $2$-vertex in $G_2$.
Since $t(o_1)+t(o_2)\ge t(o)$, we have that $v_1$ and $v_a$ are $2$-vertices, and it follows that $G_2=o_2$ is a face split off by $Q$.

\begin{figure}
\myincludegraphics{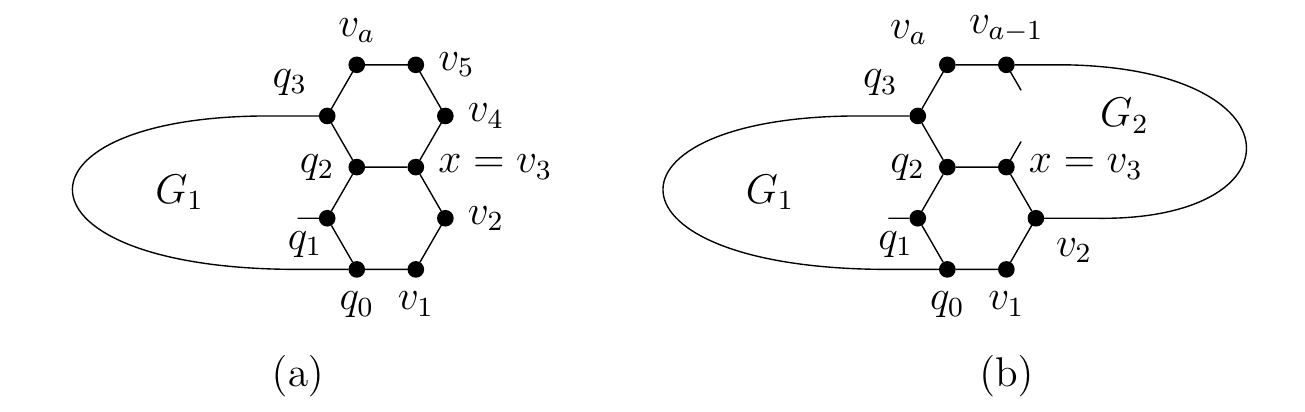}
\caption{3-chords from Lemma~\ref{lemma-nochord}. \label{fig:3chord}}
\end{figure}

Finally, suppose that $\ell=3$.  Suppose first that both $q_1$ and $q_2$ are $2$-vertices in $G_2$,
and thus $q_1q_2$ is a $33$-edge with respect to $o_1$.
Since $t(o_1)+t(o_2)\ge t(o)$, at least one of $v_1$ and $v_a$ (say $v_1$) is a $2$-vertex.
Thus, $V(Q)\cup \{v_1,v_2,v_a\}$ are all incident with a common face, which is only possible if
$v_2=v_a$ and $G_2$ consists of a single face.  It follows that $Q$ splits off a face.

The case that both $q_1$ and $q_2$ are $2$-vertices in $G_1$ is symmetrical.
Hence, without loss of generality, we assume that $q_1$ is a $2$-vertex and $q_2$ is a $3$-vertex in $G_2$.
As $t(o_1)+t(o_2)\ge t(o)$, we infer that both $v_1$ and $v_a$ are $2$-vertices.
Let $x\not\in\{q_1,q_3\}$ be the third neighbor of $q_2$.  Observe that
\begin{itemize}
\item if $x\in V(o)$, then both $xq_2q_3$ and $xq_2q_1q_0$ split off a face (for the former, note that the edge joining $v_{a-1}$ with
a neighbor of $x$ is not a chord, since we already proved that $o$ is an induced cycle). See Figure~\ref{fig:3chord}(a).
\item if $x\not\in V(o)$, then $x$ and $v_2$
are adjacent, $v_1v_2xq_2q_1q_0$ is a face and we may apply the same observations to the $3$-chord $v_2xq_2q_3$. See Figure~\ref{fig:3chord}(b).
\end{itemize}
By symmetry, this argument also holds for $o_1$. Hence by repeating the argument we conclude that $G$ is a fat worm, contradicting the assumption
that $G$ is a normal patch.

Furthermore, note that if $Q$ splits off a face, then $t(o)=t(o_1)+t(o_2)-(\ell-2)+k$, where $k$ is the number of $33$-edges incident with $q_0$ or $q_{\ell}$.
Since $G$ is indecomposable, it follows that $k\le\ell-2$.
\end{proof}

For a patch $G$ with boundary $o$, let $G'\subseteq G$ be the subgraph consisting of the outer layer of the faces of $G$;
that is, $e$ is an edge of $G'$ if and only if it is incident with a face that shares an edge with $o$.
Let $S\subseteq V(G)\setminus V(o)$ be the set of vertices that have at least two neighbors in $o$.  Let $o'=G'-(V(o)\cup S)$.
See Figure~\ref{fig:o}(a).

\begin{figure}
\myincludegraphics{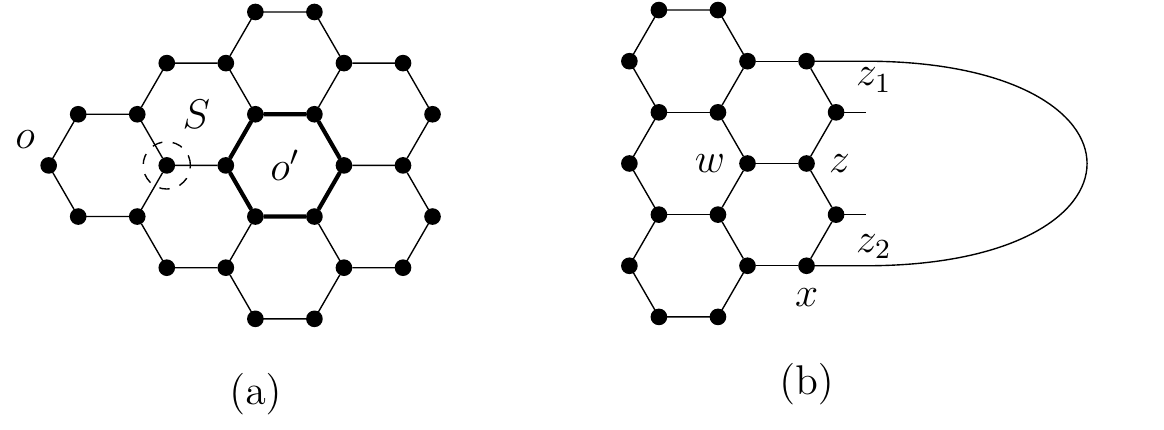}
\caption{Patch $G$ and $o'$ and a configuration from Lemma~\ref{lemma-layer}.\label{fig:o}}
\end{figure}

\begin{lemma}\label{lemma-layer}
If $G$ is a normal patch with boundary $o$, then $o'$ is a cycle, and the patch bounded by $o'$
satisfies $t(o')=t(o)$, $s(o')=s(o)$ and $s_2(o')\ge s_2(o)$.  Furthermore, 
$\ell(o')=\ell(o)+2p(G)-12-2s_2(o)$.
\end{lemma}
\begin{proof}
Since $G$ is not a fat worm, we have $|V(G)\setminus V(o)|>1$.
If two vertices of $S$ were adjacent, then $|V(G)\setminus V(o)|=2$ by Lemma~\ref{lemma-nochord} and $G$ would be a fat worm,
thus $S$ is an independent set and $o'$ is not empty.
Lemma~\ref{lemma-nochord} also implies that $G-V(o)$ is connected, and since $S$ only contains vertices whose degree in $G-V(o)$ is one,
$o'$ is connected as well.

Suppose that a vertex $w$ of $o'$ is adjacent to more than one vertex of $S$.
Since $G$ is not the shell, $w$ is adjacent to exactly two vertices in $S$; let $z$ be
the neighbor of $w$ not in $S$.  Since $G$ is not a fat worm, we have $z\not\in V(o)$.
Let $z_1$ and $z_2$ be the neighbors of $z$ distinct from $w$; since $z\not\in S$,
we may assume that $z_2\not\in V(o)$.  Let $f$ be the face of $G$ incident with $w$, $z$ and $z_2$,
and let $x$ be the neighbor of $z_2$ in $f$ distinct from $z$.  
Note that $f$ is incident with a neighbor of $w$ that belongs to $S$, and thus $f$ shares an edge with $o$.
Hence $f$ is a $6$-face and we have $x\in V(o)$.
If $z_1\in V(o)$, then the $3$-chord $z_1zz_2x$ contradicts Lemma~\ref{lemma-nochord}.
Otherwise, by a symmetric argument we conclude that a face $f'$ incident with $w$, $z$ and $z_1$
is also a $6$-face sharing an edge with $o$, see Figure~\ref{fig:o}(b).  However, $f\cup f'$ forms a simplifying
cut (a pair of $4$-chords) in $G$, which is a contradiction. Therefore, each vertex of $o'$ has at most one neighbor in $S$.

By Lemma~\ref{lemma-nochord}, no vertex of $o'$ has a neighbor both in $S$ and in $o$,
since at least one of the two resulting $3$-chords would not split off a face.
If $v$ is a vertex of $o'$ that has a neighbor in $o$ or $S$, then $v$ has two neighbors in $o'$,
and thus $o'$ has at least three vertices.

Suppose that $o'$ contains a bridge $e=uv$.  Note that both faces $f_1$ and $f_2$ of $G$ incident with $e$ share an edge with $o$.
As $u,v\not\in S$, these two vertices do not lie on $2$-chords.
Note that $f_1\cup f_2$ contains an $\ell_u$-chord $P_u$ of $o$ such that $u\in V(P_u)$ and $v\not\in V(P_u)$, where $3\le\ell_u\le 5$.
Similarly, let $P_v$ be an $\ell_v$-chord of $o$ such that $v\in V(P_v)$ and $u\not\in V(P_v)$.
As neither $P_u$ nor $P_v$ splits off a face, Lemma~\ref{lemma-nochord} implies that $\ell_u,\ell_v\ge 4$.
Since $f_1$ and $f_2$ are $6$-faces, we conclude that $P_u$ and $P_v$ are $4$-chords.
Lemma~\ref{lemma-nochord} further implies that $u$ and $v$ are middle vertices of $P_u$ and $P_v$,
thus $f_1$ and $f_2$ is a pair of $4$-chords forming a simplifying cut.  This is a contradiction;
therefore, $o'$ is $2$-edge-connected.  Since $o'\subset G'$, every edge of $o'$ is incident with a face that shares an edge
with $o$.  We conclude that $o'$ is a cycle.

Consider now a $33$-edge $x_1x_2$ in $o$ and let $x_1x_2x_3x_4x_5x_6$ be the incident $6$-face.
Lemma~\ref{lemma-nochord} implies that each of $x_3$ and $x_6$ has only one neighbor in $o$,
as otherwise one of them would belong to a $2$-chord whose endpoint is incident with a $33$-edge $x_1x_2$.
Therefore, $x_3,x_6\not\in S$ and $x_3x_4x_5x_6$ is a part of $o'$,
and $x_4x_5$ is a $33$-edge with respect to $o'$.  It follows that $t(o')\ge t(o)$.  On the other hand, consider a $33$-edge $y_4y_5$ of $o'$, and
let $y_3y_4y_5y_6$ be a part of the boundary of $o'$.  As $y_4$ and $y_5$ are $3$-vertices in $o'$, there exists a $6$-face
$y_1y_2y_3y_4y_5y_6$ in $G$, and $y_1y_2$ is a $33$-edge in $o$.  Hence, we have $t(o')=t(o)$ and by Lemma~\ref{lemma-n22}, $s(o')=s(o)$.

Similarly, consider a part $z_0z_1z_2z_3z_4z_5z_6$ of $o$, where $z_2z_3$ and $z_3z_4$ are $22$-edges.  The common neighbor $z$ of $z_1$ and $z_5$
belongs to $S$, and its neighbor $z'$ distinct from $z_1$ and $z_5$ belongs to $o'$.  As we observed before, both neighbors $z'_1$ and $z'_2$ of
$z'$ distinct from $z$ belong to $o'$.  Furthermore, by Lemma~\ref{lemma-nochord}, the endpoints of the $2$-chord $z_1zz_5$
are incident with no $33$-edges, thus $z_0$ and $z_6$ are $2$-vertices.
It follows that both $z'_1$ and $z'_2$ have a neighbor in $o$, and $z'_1z'$ and $z'_2z'$ are $22$-edges with respect to $o'$.
Hence, we conclude that $s_2(o')\ge s_2(o)$.

In fact, $D(o')$ can be obtained from $D(o)$ in the following way: Add $0$ between each two consecutive letters in $D(o)$.
Since endvertices of a $2$-chord of $o$ are not incident with $33$-edges, if $B0B$ appears in the resulting sequence, then it is as a part
of a subsequence $n_1B0Bn_2$, where $n_1,n_2\ge 3$.  We construct $D(o')$ by 
\begin{itemize}
\item for each $n_1B0Bn_2$ subsequence, decreasing each of $n_1$ and $n_2$ by $3$,
\item for each $B$ not contained in such a subsequence, decreasing each of the neighboring integers by $1$,
\item for each $A$, increasing each of the neighboring integers by $1$, and
\item suppressing any zeros.
\end{itemize}

Note that the increases/decreases are cumulative, e.g., if $D(o)$ contains a subsequence $A3B2B$, then the sequence $D(o')$
contains a subsequence $A3B0B$ (or $A3BB$ after suppressing zeros).
By Lemma~\ref{lemma-n22}, $t(o)-s(o)=p(G)-6$, and the formula for the length of $o'$ follows:
\begin{eqnarray*}
\ell(o')  &=& \ell(o) + 2t(o) - 2(s(o) - 2s_2(o)) - 6s_2(o) \\
        &=& \ell(o) + 2p(G)-12 - 2s_2(o).
\end{eqnarray*}
\end{proof}

Consider a patch $G$ with boundary $o_1$.  A sequence of cycles $o_1$, $o_2$, \ldots, $o_k$ (with $k\ge 2$) is called an {\em uninterrupted peeling}
if for $1\le i<k$, the subpatch of $G$ bounded by $o_i$ is normal and $o_{i+1}=o'_i$.

\begin{lemma}\label{lemma-peeling}
Let $o$ be the boundary of a patch $G$ such that $p(G)\neq 6$.  If $o = o_1$, $o_2$, \ldots, $o_k$ is
an uninterrupted peeling, then the number of vertices of $G$ outside of (and not including) $o_k$ is at least $4k^2/9$.
\end{lemma}

\begin{proof}
By Lemma~\ref{lemma-layer}, we have $s_2(o_1)\le s_2(o_2)\le \ldots \le s_2(o_k)$.  
Moreover, Lemma~\ref{lemma-layer} also implies that the sequence $\ell(o_1),\ldots,\ell(o_k)$ is concave.

Let $a$ be the largest index such that $\ell(o_1) < \ldots < \ell(o_{a})$ 
and let $b$ be the smallest index such that $\ell(o_b) > \ldots > \ell(o_k)$.
Note that if the whole sequence is decreasing then $a = b = 1$ and similarly if the
whole sequence is increasing then $a = b = k$, hence $a\le b$ in all the cases.
We compute a lower bound on the the number of vertices of $G$ outside of $o_k$ as
$$\sum_{i=1}^{k-1}\ell(o_i) = \sum_{i=1}^{a-1}\ell(o_i) + \sum_{i=a}^{b-1}\ell(o_i) + \sum_{i=b}^{k-1}\ell(o_i).$$


First, we deal with the middle term. Let $m = b - a$.
Suppose that $a < b$. In this case, we have $\ell(o_a)=\ell(o_{a+1})=\ldots=\ell(o_b)$; let $r=\ell(o_a)$.
By Lemma~\ref{lemma-layer}, $s_2(o_i)=p(G)-6$ for $a\le i<b$.
Since $p(G)\neq 6$, we conclude that $s_2(o_a)\ge 1$.  It follows that $D(o_a)$ contains a subsequence $n_1BBn_2$, where $n_1,n_2\ge 3$ by Lemma~\ref{lemma-nochord}.
By Lemma~\ref{lemma-n22}, $t(o_a)-s(o_a)=p(G)-6=s_2(o_a)$.  As $s(o_a)\ge 2s_2(o_a)$, we conclude that $t(o_a)\ge 3$, and thus $n_1+n_2+5\le r$.
By symmetry, assume that $2n_1\le r-5$.  As observed in the proof of Lemma~\ref{lemma-layer}, $D(o_{a+1})$ contains a subsequence $n'_1BBn'_2$,
where $n'_1\le n_1-2$ (the equality is achieved if $n_1$ is adjacent to $A$ in $D(o_a)$).  The same observation applies to $o_{a+1}$, \ldots, $o_{b-2}$.
In the normal patch $o_{b-1}$, the integers adjacent to $BB$ are greater or equal to three, thus $n_1\ge 2m+1$ and $r\ge 4m+7$.  
It follows that $\sum_{i=a}^{b-1} \ell(o_i)=mr\ge m(4m+7)$.
In the case that $a=b$, we have $\sum_{i=a}^{b-1} \ell(o_i) = 0 = m(4m+7)$, since $m = 0$.

Now we deal with the other terms of the sum.  
If $a > 1$, then the sequence $\ell(o_1),\ell(o_2),\ldots, \ell(o_{a-1})$ dominates the arithmetic sequence with the first element 
$\ell(o_1)\ge 5$ and step $2$ due to Lemma~\ref{lemma-layer} and the fact that $p(G)-6-s_2(o_i) \geq 1$ for $1 \leq i \leq a-1$. 
Hence $\sum_{i=1}^{a-1} \ell(o_i)\ge \sum_{i=1}^{a-1} (3+2i)=(a-1)(a+3)$.
If $a = 1$ then $\sum_{i=1}^{a-1} \ell(o_i) = 0 = (a-1)(a+3)$.

Similarly, the sequence $\ell(o_{b}),\ell(o_{b+1}), \ldots, \ell(o_{k-1})$
dominates the arithmetic sequence with the last element $\ell(o_{k-1})\ge 7$ and step $-2$, hence $\sum_{i=b}^{k-1} \ell(o_i)\ge \sum_{i=1}^{k-b} (5+2i)=(k-b)(k-b+6)$.

Note that $(a-1)+m+(k-b)=k-1$.  Summing these inequalities, we obtain
\begin{eqnarray*}
\sum_{i=1}^{k-1} \ell(o_i)&\ge&(a-1)(a+3)+m(4m+7)+(k-b)(k-b+6)\\
&\ge&(a-1)^2+4m^2+(k-b)(k-b+2)+1\\
&\ge& 4k^2/9,
\end{eqnarray*}
where the lower bound in the last inequality is achieved for $a-1=4k/9$, $k-b=4k/9-1$ and $m=k/9$.
Since all the cycles $o_1$, \ldots, $o_{k-1}$ are strictly outside of $o_k$, the claim follows.
\end{proof}

\begin{lemma}\label{lemma-dist}
Let $H$ be a fullerene with $n$ vertices and $f$ a $5$-face of $H$.  There exist at least five $5$-faces distinct from $f$ whose distance to $f$ in the dual of $H$
is at most $\sqrt{63n/2}+14$.
\end{lemma}
\begin{proof}
We define a rooted tree $T$ with each vertex of $T$ corresponding to a patch $G\subseteq H$ such that $p(G)\neq 0$ and $p(G)\neq 6$.
Furthermore, we assign a weight $d(e)$ to each edge $e$ of $T$.  The root of $T$ is the patch $G_0=H$ whose boundary is the cycle
bounding $f$, i.e., $p(G_0)=11$.  Suppose that a patch $G$ with boundary $o$ is a vertex of $T$.
Let us note that $G$ is neither a worm nor the shell, since $p(G)>0$. The sons of $G$ in the tree are defined as follows:
\begin{itemize}
\item[(a)] If $p(G)\in \{1,7\}$ and $o$ shares an edge with a $5$-face of $G$, then $G$ is a leaf of $T$.
\item[(b)] If $G$ is a normal patch, then $G$ has a single son $G'$, equal to the last element of the maximal uninterrupted peeling starting with $o$.
The weight of the the edge $e$ joining $G$ with $G'$ is equal to the length (number of patches) of the uninterrupted peeling.
Note that $G'$ is not a normal patch.
\item[(c)] If $G$ has a simplifying cut, then let $o_1$ and $o_2$ be the boundaries of the two patches $G_1$ and $G_2$ to that it splits $G$.
Note that $t(o_1)+t(o_2)<t(o)$.  The patch $G_i$ (with $i\in \{1,2\}$) is a son of $G$ if $p(G_i)\neq 0$ and $p(G_i)\neq 6$.
In that case, the edge between $G$ and $G_i$ has weight $1$.  Since $0<p(G)<12$ and $p(G)\neq 6$, $G$ has at least one son.
\item[(d)] Finally, if $G$ is indecomposable, $p(G)\not\in \{1,7\}$ and $o$ shares an edge with a $5$-face $f'$, note that there exists an $\ell$-chord (with $\ell\le 4$)
splitting off $f'$ (otherwise $f'$ would be incident with a chord and a $2$-chord and both of them would witness the decomposability of $G$).
We let the son $G'$ of $G$ with boundary $o'$ be the patch obtained from $G$ by removing edges incident to both $f'$ and $o$ and 
by removing isolated vertices.
We let the edge of $T$ between $G$ and $G'$ have weight $1$.
\end{itemize}
The \emph{type} of $G$ is defined according to the rule ((a) to (d)) in that its sons are described.

Observe that at least five $5$-faces distinct from $f$ share edges with boundaries of the patches
forming the vertices of $T$ of type (a) or (d).  Indeed, either all $5$-faces are reachable in this way,
or there are exactly six potentially unreachable $5$-faces contained in a single patch that is a leaf of $T$,
or split off by a simplifying cut from an internal vertex of $T$.  Let $T_1$ be a subtree of $T$ of smallest
possible depth that contains five vertices of type (a) or (d).  We choose $T_1$ to be minimal, i.e., all leaves
of $T_1$ are of type (a) or (d).

Consider a vertex $G_1$ with a son $G_2$ in $T_1$, and let $o_1$ and $o_2$ be the boundaries of these patches.
If $G_1$ is of type (b), then $p(G_1)=p(G_2)$ and $t(o_1)=t(o_2)$ by Lemma~\ref{lemma-layer}.  If $G_1$ is
of type (c), then $p(G_1)\ge p(G_2)$ and $t(o_1)>t(o_2)$.  If $G_1$ is of type (d), then
$p(G_1)>p(G_2)$ and $t(o_1)\ge t(o_2)-1$.

Let $P=p_1p_2\ldots p_m$ be the path in $T$ joining the root $G_0=p_1$ with a leaf $G=p_m$ whose boundary is incident with a $5$-face $f'$.
Observe that the distance between $f$ and $f'$ in the dual of $G$ is at most the sum of the weights of the edges of $P$, plus $1$.  Let $o_0$ and $o$ be the boundaries
of $G_0$ and $G$, respectively.  Let $m_b$, $m_c$ and $m_d$ be the numbers of vertices of types (b), (c) and (d) in $P$ distinct from $G$, respectively.
By the observations in the previous paragraph, we have $t(o)\le t(o_0)+m_d-m_c=5+m_d-m_c$.  By the choice of $T_1$, we have $m_d\le 4$,
and since $t(o)$ is nonnegative, $m_c\le 9$.  Therefore,
$$\sum_{\mbox{$p_i$ is non-normal}} d(p_ip_{i+1})\le m_c+m_d\le 13.$$
Let $d_1$, $d_2$, \ldots, $d_{m_b}$ be the sequence of the weights of all edges $p_ip_{i+1}$ of $P$ such that $p_i$ is a normal patch;
by the construction of $T$, $p_{i+1}$ is not a normal patch in this case, hence $m_b\le m_c+m_d+1\le 14$.
Using Lemma~\ref{lemma-peeling}, we obtain
$$n\ge \frac{4}{9}\sum_{i=1}^{m_b}d_i^2\ge\frac{4}{9m_b}\left(\sum_{i=1}^{m_b}d_i\right)^2.$$
Therefore, the total weight of these edges is at most $\sqrt{63n/2}$, and the distance between $f$ and $f'$ is at most
$\sqrt{63n/2}+14$.
\end{proof}

\begin{lemma}\label{lemma-match}
Every graph $G$ on $12$ vertices with minimum degree $5$ such that $K_{5,7}\not\subseteq G$ has a perfect matching.
\end{lemma}
\begin{proof}
If $G$ does not have a perfect matching, then there exists a set $S\subseteq V(G)$ such that
$G-S$ has more than $|S|$ components of odd size.  Consider such a set $S$, and observe that $|S|<6$.
As $\delta(G)\ge 5$, $G$ is either $2K_6$ (and thus has a perfect matching) or $G$ is connected.  Therefore, $|S|\ge 1$.

If $|S|<5$, then since $\delta(G)\ge 5$, no component of $G-S$ may consist of a single vertex, and hence $G-S$ has at most three odd components
and $|S|\le 2$.  Since $\delta(G)\ge 5$, each component of $G-S$ has size at least $4$.  However, $G-S$ must have at least two components of
odd size, thus it would have exactly two components of size $5$.  However, then $|S|=2$, which is not smaller than the number of odd components.

Therefore, $|S|=5$ and $G-S$ has at least $6$ components of odd size.  However, this is only possible if each component of
$G-S$ consists of a single vertex, and hence $K_{5,7}\subseteq G$.
\end{proof}

\begin{proof}[Proof of Theorem~\ref{thm-main}]
Let $K'_H$ be the subgraph of $K_H$ consisting of edges with weight at most $\sqrt{63n/2}+14$.  By Lemma~\ref{lemma-dist},
$\delta(K'_H)\ge 5$, and thus $K'_H$ either has a perfect matching or $K_{5,7}$ as a subgraph, by Lemma~\ref{lemma-match}.
In the former case, the weight of each perfect matching in $K'_H$ (and thus of the minimum-weight perfect matching in $K_H$)
is at most $6(\sqrt{63n/2}+14) = \sqrt{1134n}+84$.   In the latter case, note that the weights in $K_H$ satisfy the triangle inequality, thus
the weight of any edge in $K_H$ is at most $2(\sqrt{63n/2}+14)$, and we conclude that $K_H$ has a perfect matching
of weight at most $(5+2)(\sqrt{63n/2}+14) = \sqrt{3087n/2}+98$.  By Theorem~\ref{thm-bibmatch}, $b(H)=O(\sqrt{n})$.
\end{proof}

The multiplicative constant $\sqrt{3087/2}\approx 39.29$ is likely to be far from the best possible.  Indeed, it can be somewhat
improved by a more complicated analysis of our argument (e.g., observing that not all $5$-faces can appear in $T$ on the lowest
possible level, indicating that some of the edges of $K_H$ are much shorter than we estimated).  Nevertheless, we could not
improve it enough to approach the best known lower bound of $\sqrt{12/5}\approx 1.549$ of Do\v{s}li\'c and Vuki\v{c}evi\'c~\cite{DV}.


\begin{thebibliography}{11}

   \bibitem{D} T.~Do\v{s}li\'c,
      {\em Bipartivity of fullerene graphs and fullerene stability}, 
          Chemical Physics Letters {\bf 412}  (2005), 336--340.

   \bibitem{DV} T.~Do\v{s}li\'c, D.~Vuki\v{c}evi\'c,
      {\em Computing the bipartite edge frustration of fullerene graphs}, 
          Discrete Applied Mathematics {\bf 155}  (2007), 1294--1301.

   \bibitem{DDE}  M.~S.~Dresselhaus, G.~Dresselhaus, P.~C.~Eklund, 
      {\em Science of Fullerenes and Carbon Nanotubes}, 
          Academic Press, New York (1996). 


   \bibitem{FL} S.~Fajtlowicz, C.~E.~Larson,
      {\em Graph-theoretic independence as a predictor of fullerene stability},
        Chemical Physics Letters {\bf 377} (2003), 485--490.

   \bibitem{FM} P.~W.~Fowler, D.~E.~Manolopoulos, 
      {\em An Atlas of Fullerenes}, 
          Oxford Univ. Press, Oxford (1995).
  
   \bibitem{GC}  I.~Gutman, S.~J.~Cyvin, 
      {\em Introduction to the Theory of Benzenoid Hydrocarbons},
         Springer-Verlag, Berlin (1989). 

   \bibitem{KKMS} F.~Kardo\v{s}, D. Kr\'al', J. Mi\v{s}kuf, J.-S. Sereni,
      {\em Fullerene graphs have exponentially many perfect matchings}, 
        Journal Mathematical Chemistry {\bf 46} (2009), 443--447. 

   \bibitem{KS} F.~Kardo\v{s}, R.~\v Skrekovski, 
      {\em Cyclic edge-cuts in fullerene graphs}, 
        Journal Mathematical Chemistry {\bf 22} (2008), 121--132. 

   \bibitem{KLFH} W.~Kr\"{a}tschmer, L.~D.~Lamb, K.~Fostiropoulos, D.~R.~Huffman, 
       {\em Solid $C_{60}$: a new form of carbon}, Nature {\bf 347} (1990) 354--358.
  
   \bibitem{KHBCS} H.~W.~Kroto, J.~R.~Heath, S.~C.~O'Brien, R.~F.~Curl, R.~E.~Smalley, 
       {\em C$_{60}$ Buckminsterfullerene}, Nature {\bf 318} (1985) 162--163.
  

   \bibitem{M} J.~Malkevitch, {\em Geometrical and Combinatorial Questions About Fullerenes}, in: P.~Hansen, P.~Fowler, M.~Zheng (Eds.), 
        {Discrete Mathematical Chemistry}, DIMACS Series in Discrete Mathematics and Theoretical Computer Science {\bf 51} (2000)  261--266.
  
   \bibitem{THAK} R.~Taylor, J.~P.~Hare, A.~K.~Abdul-Sada, H.~W.~Kroto, 
        {\em Isolation, separation and characterisation of the fullerences $C_{60}$ and $C_{70}$: the third form of carbon}, 
        Journal of the Chemical Society, Chemical Communications  (1990) 1423--1425.
    
\end{thebibliography}
\end{document}